\pdfoutput=1
\RequirePackage{ifpdf}
\ifpdf 
\documentclass[pdftex]{sigma}
\else
\documentclass{sigma}
\fi

\def\spn{\mathop{\rm span}}
\def\Hom{\operatorname{Hom}}
\def\ind{\operatorname{ind}}

\def\phi{\varphi}
\def\g{\mathfrak g}
\def\m{\mathfrak m}
\def\V{\mathfrak V}
\def\F{\mathbb F}
\def\N{\mathbb N}

\numberwithin{equation}{section}

\newtheorem{Theorem}{Theorem}[section]
\newtheorem{Lemma}[Theorem]{Lemma}
\newtheorem{Proposition}[Theorem]{Proposition}
 { \theoremstyle{definition}
\newtheorem{Remark}[Theorem]{Remark} }

\begin{document}

\allowdisplaybreaks

\newcommand{\arXivNumber}{1901.07532}

\renewcommand{\thefootnote}{}

\renewcommand{\PaperNumber}{095}

\FirstPageHeading

\ShortArticleName{Cohomology of Restricted Filiform Lie Algebras ${\mathfrak m}_2^\lambda(p)$}

\ArticleName{Cohomology of Restricted Filiform Lie Algebras $\boldsymbol{{\mathfrak m}_2^\lambda(p)}$\footnote{This paper is a~contribution to the Special Issue on Algebra, Topology, and Dynamics in Interaction in honor of Dmitry Fuchs. The full collection is available at \href{https://www.emis.de/journals/SIGMA/Fuchs.html}{https://www.emis.de/journals/SIGMA/Fuchs.html}}}

\Author{Tyler J.~EVANS~$^\dag$ and Alice FIALOWSKI~$^{\ddag\S}$}

\AuthorNameForHeading{T.J.~Evans and A.~Fialowski}

\Address{$^\dag$~Department of Mathematics, Humboldt State University, Arcata, CA 95521, USA}
\EmailD{\href{mailto:evans@humboldt.edu}{evans@humboldt.edu}}
\URLaddressD{\url{https://sites.google.com/humboldt.edu/tylerjevans}}

\Address{$^\ddag$~Institute of Mathematics, University of P\'ecs, P\'ecs, Hungary}
\EmailD{\href{mailto:fialowsk@ttk.pte.hu}{fialowsk@ttk.pte.hu}}

\Address{$^\S$~Institute of Mathematics E\"otv\"os Lor\'and University, Budapest, Hungary}
\EmailD{\href{mailto:fialowsk@cs.elte.hu}{fialowsk@cs.elte.hu}}

\ArticleDates{Received August 19, 2019, in final form November 24, 2019; Published online December 01, 2019}

\Abstract{For the $p$-dimensional filiform Lie algebra ${\mathfrak m}_2(p)$ over a field ${\mathbb F}$ of prime characteristic $p\ge 5$ with nonzero Lie brackets $[e_1,e_i] = e_{i+1}$ for $1<i<p$ and $[e_2,e_i]=e_{i+2}$ for $2<i<p-1$, we show that there is a family ${\mathfrak m}_2^{\lambda}(p)$ of restricted Lie algebra structures parameterized by elements $\lambda \in {\mathbb F}^p$. We explicitly describe bases for the ordinary and restricted 1- and 2-cohomology spaces with trivial coefficients, and give formulas for the bracket and $[p]$-operations in the corresponding restricted one-dimensional central extensions.}


\Classification{17B50; 17B56}

\begin{flushright}\begin{minipage}{70mm}
\it We dedicate this paper to Dmitry B.~Fuchs\\ on the occasion of his 80th birthday
\end{minipage}
\end{flushright}

\renewcommand{\thefootnote}{\arabic{footnote}}
\setcounter{footnote}{0}

\section{Introduction}
$\N$-graded Lie algebras of maximal class have been intensively studied in the last decade. A Lie algebra of \emph{maximal class} is
a graded Lie algebra
\[
 \g=\oplus_{i=1}^{\infty}\g_i
\]
over a field $\F$, where $\dim(\g_1)=\dim(\g_2)=1$, dim$(\g_i) \leq 1$
for $i \ge 3$ and $[\g_1,\g_i]= \g_{i+1}$ for $i \ge 1$.

A Lie algebra of dimension $n$ is called \emph{filiform} if
\[\dim\big(\g^k\big)=n-k,\qquad 2 \le k \le n,\qquad {\rm where} \quad \g^k=\big[\g,\g^{k-1}\big].\]

Lie algebras of maximal class with two generators over fields of characteristic zero have been classified, and exactly three of these algebras are of filiform type \cite{Fialowski1983}. We list them with the nontrivial bracket structures:
\begin{alignat*}{4}
& \m_0\colon \quad&& [e_1 ,e_i]= e_{i+1}, \qquad && i \ge 2,&\\
& \m_2\colon \quad&& [e_1, e_i]= e_{i+1}, \qquad && i \ge 2,&\\
&&& [e_2, e_j] = e_{j+2}, \qquad && j \ge 3, &\\
& \V\colon \quad && [e_i, e_j] = (j-i)e_{i+j}, \qquad && i, j \ge 1.&
\end{alignat*}

Filiform $\N$-graded Lie algebras $\g$ of dimension $n$ over a field of characteristic zero that satisfy $[\g_1,\g_i]=\g_{i+1}$ and $\dim(\g_i)=1$ for $i<n$ (which is equivalent to having 2 generators) are classified in~\cite{Millionschikov2004}. They include the natural ``truncations'' of $\m_0(n)$ and $\m_2(n)$ obtained by taking the quotient by the ideal generated by $e_{n+1}$. The algebra $\V$ (the Witt algebra) is isomorphic to the algebra of derivations of the polynomial algebra $\F[x]$. If $\F$ has characteristic $p>0$, then the truncation $\V(p)$ of $\V$ is the derivation algebra of the quotient of $\F[x]$ by the ideal generated by $x^p-1$. The algebra $\V(p)$ is called the (modular) Witt algebra.

The above picture is more complicated in the modular case (that is,
over fields of positive characteristic), see \cite{CMN1997, CN2000, Jurman2005}, but
$\m_0$, $\m_2$, $\V$ and their truncations always show up. We refer the
reader to the book~\cite{SF} for a general treatment of modular Lie
algebras. In this paper, we show that if the field $\F$ has
characteristic $p\ge 5$, then the Lie algebra $\m_2(p)$ admits a
family of restricted Lie algebra structures $\m_2^{\lambda}(p)$
parameterized by elements $\lambda \in \F^p$. We describe the
isomorphism classes of these algebras, calculate the ordinary and restricted cohomology spaces with trivial coefficients $H^q\big(\m_2^{\lambda}(p)\big)$ and $H_*^q\big(\m_2^{\lambda}(p)\big)$ for $q=1,2$ and give explicit bases for those spaces. We also give the bracket
structures and $[p]$-operations for the corresponding restricted
one-dimensional central extensions of these restricted Lie algebras.

With this, we complete the description of all three types of truncated filiform restricted Lie algebras ($\m_0^\lambda(p)$, $\m_2^\lambda(p)$, and $\V(p)$), their low dimensional cohomology spaces with trivial coefficients and their restricted one-dimensional central extensions. The algebras $\m_0^{\lambda}(p)$ were studied in~\cite{EvFi2019}, and the algebra $\V(p)$ was studied in~\cite{EvFiPe2016} (where it is denoted by~$W$).

\begin{Remark}
 For $p=2$ and $p=3$, $\m_0(p)=\m_2(p)$ so these algebras were
 treated in \cite{EvFi2019}.
\end{Remark}

In this paper, all cochain and cohomology spaces are with coefficients in the trivial $\F$-module. The organization is as follows. In Section~\ref{section2} we construct the restricted Lie algebra family~$\m_2^{\lambda}(p)$, determine the isomorphism classes of these restricted Lie algebras, and describe both the ordinary and restricted 1- and 2-cochains, including formulas for all differentials. In Section~\ref{section3} we calculate both the ordinary and restricted 1-cohomology by giving explicit cocycles. Section~\ref{section4} contains the calculation of the ordinary and restricted 2-cohomology spaces, again by giving explicit cocycles. In Section~\ref{section5} we describe all restricted one-dimensional central extensions and give their brackets and $[p]$-operations.

\section{Preliminaries}\label{section2}

\subsection[The Lie algebra $\m_2(p)$]{The Lie algebra $\boldsymbol{\m_2(p)}$}\label{section2.1}

Let $p\ge 5$ be a prime, and let $\F$ be a field of characteristic
$p$. Define the $\F$-vector space
\[\m_2(p)=\spn\nolimits_\F(\{e_1,\dots , e_p\}), \]
and define a bracket on $\m_2(p)$ by \
\begin{gather*}
 [e_1,e_i] =e_{i+1},\qquad 1<i < p,\\
 [e_2,e_i] =e_{i+2},\qquad 2<i < p-1,
\end{gather*}
with all other brackets $[e_i,e_j]$ (for $i<j$) being $0$. Note that
$\m_2(p)$ is a graded Lie algebra with $k$-th graded component
$(\m_2(p))_k=\F e_k$ for $1\le k\le p$. If $\alpha_i, \beta_i \in \F$
and $g=\sum_{i=1}^p \alpha_ie_i$, $h=\sum_{i=1}^p \beta_ie_i$, then
\begin{gather} [g,h] =(\alpha_1\beta_2-\alpha_2\beta_1)e_3+(\alpha_1\beta_3-\alpha_3\beta_1)e_4 \nonumber\\
\hphantom{[g,h] =}{} +\sum_{j=5}^p ((\alpha_1\beta_{j-1}-\alpha_{j-1}\beta_1)+(\alpha_2\beta_{j-2}-\alpha_{j-2}\beta_2)) e_j.\label{bracket}
\end{gather}

\subsection[The restricted Lie algebras $\m_2^\lambda(p)$]{The restricted Lie algebras $\boldsymbol{\m_2^\lambda(p)}$}

We refer the reader to \cite[Chapter~V, Section~7]{JacobsonBook} and \cite[Section~2.2]{SF} for the definition of a~restricted Lie algebra,
and for the construction of the $[p]$-mapping on a given Lie algebra
($\m_2(p)$ in the current paper) used below. For any $j\ge 2$ and
$g_1,\dots, g_j\in\m_2(p)$, we denote the $j$-fold bracket
\[[g_1,g_2,g_3,\dots, g_j]=[[\dots[[g_1,g_2],g_3],\dots],g_j].\] Since
$p\ge 5$, (\ref{bracket}) implies that the center of the algebra is
$Z(\m_2(p))= \F e_p,$ and $p$-fold brackets are zero. Therefore for
each $\lambda=(\lambda_1, \dots , \lambda_p)\in\F^p$, setting
$e_k^{[p]}=\lambda_k e_p$ for each $k$ defines a~restricted Lie
algebra that we denote by $\m_2^\lambda(p)$. Because $p$-fold
brackets in $\m_2^\lambda(p)$ are zero, for all
$g,h\in\m_2^\lambda(p)$, $\alpha\in\F$,
\[(g+h)^{[p]}=g^{[p]}+h^{[p]}\qquad {\rm and}\qquad (\alpha g)^{[p]}=\alpha^p g^{[p]},\] and therefore the $[p]$-mapping on
$\m_2^\lambda(p)$ is $p$-semilinear (see also \cite[Chapter~2, Lemma~1.2]{SF}). From this we get that if
$g=\sum \alpha_k e_k\in\m_2^\lambda(p)$, then
\begin{eqnarray}\label{p-op} g^{[p]}=\left (\sum_{k=1}^p \alpha_k^p\lambda_k \right ) e_p.
\end{eqnarray}
Everywhere below, we write $\m_2^\lambda(p)$ to denote both the graded
Lie algebra $\m_2(p)$ and the graded restricted Lie algebra
$\m_2^\lambda(p)$ for a given $\lambda\in\F^p$. The Lie brackets and
restricted $[p]$-operators for these algebras are explicitly given by~(\ref{bracket}) and~(\ref{p-op}), respectively.

\begin{Remark} For $p=2$ there are several other possible $[2]$-mappings, namely any 2-semilinear transformation on $\m_2(2)$.
\end{Remark}

\subsection{Isomorphism classes}

\begin{Proposition} Let $p\geq 5$. If $\lambda,\lambda'\in\F^p$, the graded restricted Lie algebras~$\m_2^{\lambda}(p)$ and~$\m_2^{\lambda'}(p)$ are isomorphic if and only if there exists a non-zero $\mu\in\F$ such that $\lambda_k=\mu^{(k-1)p}\lambda'_k$ for $k=1,\dots,p$.
\end{Proposition}

\begin{proof} We only consider isomorphisms that preserve the grading as we are
 interested in these algebras as graded restricted Lie
 algebras. Assume that there exists a graded restricted Lie algebra
 isomorphism $\phi\colon \m_2^{\lambda}(p) \to \m_2^{\lambda'}(p)$, and
 let $\phi(e_1)=\mu e_1, \phi(e_2)=\nu e_2$ for some non-zero
 $\mu, \nu \in \F$. Since $\phi$ preserves the Lie bracket, we must
 have $\phi(e_3)=\mu\nu e_3$, $\phi(e_4)=\mu^2\nu e_4$,
 $\phi(e_5)=\mu^3\nu e_5$. On the other hand, as
 $[e_1, e_4] = [e_2,e_3]$, we also must have
 $\phi(e_5)=\mu\nu^2 e_5$. From this it follows that $\nu=\mu^2$ and
 $\phi(e_k)=\mu^k e_k$ for $k=1,\dots, p$.

 Moreover, $\phi$ preserves the restricted $[p]$-structure so that
 \[
 \phi(e_k^{[p]})=\phi(e_k)^{[p]'}
 \]
 for $k=1,\dots, p$ (here $[p]'$ denotes the restricted $[p]$-structure on
 $\m_2^{\lambda'}(p)$). Now,
 \[
 \phi(e_k^{[p]})=\phi(\lambda_k e_p)=\lambda_k\mu^p e_p\qquad {\rm and}\qquad \phi(e_k)^{[p]'}=(\mu^k e_k)^{[p]'}=\mu^{kp}\lambda_k' e_p
 \]
 so $\lambda_k\mu^p=\mu^{kp}\lambda'_k,$ and hence
 \[
 \lambda_k=\mu^{(k-1)p}\lambda'_k.
 \]

It remains to show that the above condition on $\lambda_k$ gives rise to a graded restricted Lie algebra isomorphism between $\m_2^{\lambda}(p)$ and $\m_2^{\lambda'}(p)$. If, for $0\ne \mu\in\F$, we define $\phi(e_1)=\mu e_1$, $\phi(e_k)=\mu^k e_k$ $(2\le k\le p)$, then it is easy to check that the argument above is reversible, and we obtain a~graded isomorphism between the restricted Lie algebras.
\end{proof}

\subsection{Cochain complexes with trivial coefficients}

For the convenience of the reader and to establish our notations, we
briefly recall the definitions of the cochain spaces used below to
compute both the ordinary and restricted Lie algebra 1- and
2-cohomology. The reader can find more details on these complexes in
\cite{ChevalleyEilenberg1948,EvFi2017,EvFi2019,FuchsBook,Hochschild1954}.

\subsubsection{Ordinary cochain complex}

For ordinary Lie algebra cohomology with trivial coefficients, the relevant cochain spaces from the Chevalley--Eilenberg complex (with bases) for our purposes are
\begin{alignat*}{3}
& C^0\big(\m_2^\lambda(p)\big)= \F, \qquad && \{1\},&\\
& C^1\big(\m_2^\lambda(p)\big)= \m_2^\lambda(p)', \qquad && \big\{e^k\, |\, 1\le k\le p\big\},&\\
& C^2\big(\m_2^\lambda(p)\big)= \big({\wedge}^2\m_2^\lambda(p)\big)',\qquad && \big\{e^{i,j}\, |\, 1\le i<j\le p\big\},&\\
& C^3\big(\m_2^\lambda(p)\big)= \big({\wedge}^3\m_2^\lambda(p)\big)', \qquad && \big\{e^{s,t,u}\, |\, 1\le s<t<u\le p\big\},&
\end{alignat*}
($V'$ denotes the dual vector space) and the differentials are defined by
\begin{alignat*}{3}
& {\rm d}^0\colon \ C^0\big(\m_2^\lambda(p)\big)\to C^1\big(\m_2^\lambda(p)\big), \qquad && {\rm d}^0=0,&\\
& {\rm d}^1\colon \ C^1\big(\m_2^\lambda(p)\big)\to C^2\big(\m_2^\lambda(p)\big), \qquad && {\rm d}^1(\psi)(g,h)=\psi([g,h]),&\\
& {\rm d}^2\colon \ C^2\big(\m_2^\lambda(p)\big)\to C^3\big(\m_2^\lambda(p)\big), \qquad && {\rm d}^2(\phi)(g,h,f)=\phi([g,h]\wedge f)-\phi([g,f]\wedge h)&\\
&&& \hphantom{{\rm d}^2(\phi)(g,h,f)=}{} +\phi([h,f]\wedge g).&
\end{alignat*}
The cochain spaces $C^n\big(\m_2^\lambda(p)\big)$ are graded
\begin{alignat*}{3}
& C^1_k\big(\m_2^\lambda(p)\big)=\spn\big(\big\{e^k\big\}\big),\qquad && 1\le k\le p, &\\
& C^2_k\big(\m_2^\lambda(p)\big)=\spn\big(\big\{e^{i,j}\big\}\big),\qquad && 1\le i<j \le p,\quad i+j=k, \quad 3\le k\le 2p-1, &\\
& C^3_k\big(\m_2^\lambda(p)\big)=\spn\big(\big\{e^{s,t,u}\big\}\big),\qquad && 1\le s<t<u \le p, \quad s+t+u=k, \quad 6\le k\le 3p-3,&
\end{alignat*}
and the differentials are graded maps. If we adopt the convention that $e^{i,j}=0$ whenever $j\le i$, we can write for $1\le k\le p$
\begin{gather}\label{1-coboundary}
 {\rm d}^1\big(e^k\big)=e^{1,k-1}+e^{2,k-2}.
\end{gather}

Using the convention that $e^{i,j,k}=0$ unless $i<j<k$, we can write
\begin{gather}
 {\rm d}^2\big(e^{1,j}\big) = -e^{1,2,j-2}, \quad 2\le j,\nonumber\\
 {\rm d}^2\big(e^{i,j}\big) =e^{1,i-1,j}+e^{1,i,j-1}+e^{2,i-2,j}+e^{2,i,j-2}, \qquad 2\le i<j\le p. \label{2-coboundary}
\end{gather}

\subsubsection{Restricted cochain complex}

The relevant restricted cochain spaces are
\begin{gather*}
 C^0_*\big(\m_2^\lambda(p)\big)= C^0\big(\m_2^\lambda(p)\big),\\
 C^1_*\big(\m_2^\lambda(p)\big)= C^1\big(\m_2^\lambda(p)\big),\\
 C^2_*\big(\m_2^\lambda(p)\big)=\big\{(\phi,\omega)\, |\, \phi\in C^2(\m_2^\lambda(p)),\, \omega\colon \m_2^\lambda(p)\to\F\\
\hphantom{C^2_*\big(\m_2^\lambda(p)\big)=\big\{}{} \text{\rm has the $*$-property with respect to $\phi$}\big\}\\
 C^3_*\big(\m_2^\lambda(p)\big)=\big\{(\zeta,\eta)\, |\, \zeta\in C^3(\m_2^\lambda(p)),\, \eta\colon \m_2^\lambda(p)\times\m_2^\lambda(p)\to\F \big\}.
\end{gather*}

We recall that if $\phi\in C^2\big(\m_2^\lambda(p)\big)$, then a map $\omega\colon \m_2^\lambda(p)\to\F$ has the {\it $*$-property with respect to $\phi$} if for all $\alpha\in\F$ and all $g,h\in \m_2^\lambda(p)$ we have $\omega(\alpha g)=\alpha^p\omega(g)$ and
\begin{gather} \label{starprop}
 \omega (g+h)=\omega (g)+ \omega (h) + \sum_{\substack{g_i=g\ {\rm or}\ h\\ g_1=g,\, g_2=h}} \frac{1}{\#(g)} \phi([g_1,g_2,\dots, g_{p-1}]\wedge g_p).
\end{gather} Here $\#(g)$ is the number of factors $g_i$ equal to $g$. Moreover, given $\phi$, we can assign the values of $\omega$ arbitrarily on a basis for $\m_2^\lambda(p)$ and use~(\ref{starprop}) to define $\omega\colon \m_2^\lambda(p)\to\F$ that has the $*$-property with respect to $\phi$ (see \cite[pp.~249--250]{EvFi2019}). Recall the space of {\it Frobenius homomorphisms} $\Hom_{\rm Fr}(V,W)$ from the $\F$-vector space $V$ to the $\F$-vector space $W$ is defined by
\[ \Hom_{\rm Fr}(V,W) = \big\{ f\colon V\to W\, |\, f(\alpha x+ \beta y) = \alpha^p f(x) + \beta^p f(y)\big\}\] for all $\alpha,\beta\in\F$ and
$x,y\in V$. A map $\omega\colon \m_2^\lambda(p)\to\F$ has the $*$-property with respect to $\phi=0$ if and only if $\omega\in\Hom_{\rm Fr}(\m_2^\lambda(p),\F)$. In particular, if
$1\le k\le p$, then the map $\overline e^k\colon \m_2^\lambda(p)\to\F$ defined by
\[\overline e^k\left (\sum_{i=1}^p \alpha_ie_i\right ) = \alpha_k^p,\]
has the $*$-property with respect to $0$.

We will use the following bases for the restricted cochains
\begin{alignat*}{3}
& C^0_*\big(\m_2^\lambda(p)\big) \qquad && \{1\},&\\
& C^1_*\big(\m_2^\lambda(p)\big)\qquad & & \big\{e^k\, |\, 1\le k\le p\big\},&\\
& C^2_*\big(\m_2^\lambda(p)\big)\qquad & & \big\{\big(e^{i,j},\tilde e^{i,j}\big)\, |\, 1\le i<j\le p\big\}\cup \big\{\big(0,\overline e^k\big)\, |\, 1\le k\le p\big\},&
\end{alignat*}
where $\tilde e^{i,j}$ is the map $\tilde e^{i,j}\colon \m_2^\lambda(p)\to \F$ that vanishes on the basis and has the $*$-property with respect to $e^{i,j}$. More generally, given $\phi\in C^2(\m_2^\lambda(p))$, let $\tilde\phi\colon \m_2^\lambda(p) \to\F$ be the map that vanishes on the basis for $\m_2^\lambda(p) $ and has the $*$-property with respect to~$\phi$. The restricted differentials are defined by
\begin{alignat*}{3}
& {\rm d}^0_*\colon \ C^0_*\big(\m_2^\lambda(p)\big)\to C^1_*\big(\m_2^\lambda(p)\big) \qquad && {\rm d}^0_*=0,&\\
& {\rm d}^1_*\colon \ C^1_*\big(\m_2^\lambda(p)\big)\to C^2_*\big(\m_2^\lambda(p)\big) \qquad && {\rm d}^1_*(\psi)=\big({\rm d}^1(\psi),\ind^1(\psi)\big),&\\
& {\rm d}^2_*\colon \ C^2_*\big(\m_2^\lambda(p)\big)\to C^3_*\big(\m_2^\lambda(p)\big) \qquad && {\rm d}^2_*(\phi,\omega)=\big({\rm d}^2(\phi),\ind^2(\phi,\omega)\big),&
\end{alignat*}
where $\ind^1(\psi)(g) := \psi\big(g^{[p]}\big)$ and $\ind^2(\phi,\omega)(g,h) := \phi\big(g\wedge h^{[p]}\big)$. If $\psi\in C^1_*\big(\m_2^\lambda(p)\big)$ and $(\phi,\omega)\in C^2_*\big(\m_2^\lambda(p)\big)$, then $\ind^1(\psi)$ has the $*$-property with respect to~${\rm d}^1(\psi)$ \cite[Lemma~4]{EvansFuchs2008}. If $g=\sum\alpha_i e_i$, $h=\sum\beta_i e_i$, $\psi=\sum \mu_ie^i$ and $\phi=\sum \sigma_{ij}e^{i,j}$, then
\begin{gather*}
 \ind^1(\psi)(g) = \mu_p \left (\sum_{j=1}^p \alpha_j^p\lambda_j \right )
\end{gather*}
and
\begin{gather}\label{induced2}
 \ind^2(\phi,\omega)(g,h) = \left (\sum_{i=1}^p \beta_i^p\lambda_i \right )\left (\sum_{j=1}^{p-1} \alpha_j\sigma_{jp}\right).
\end{gather}

\begin{Remark}\label{changecomp2}
 For a given $\phi\in C^2\big(\m_2^\lambda(p)\big)$, if $(\phi,\omega),(\phi,\omega')\in C^2_*\big(\m_2^\lambda(p)\big)$, then ${\rm d}^2_*(\phi,\omega)={\rm d}^2_*(\phi,\omega')$. In particular, with trivial coefficients, $\ind^2(\phi,\omega)$ depends only on $\phi$.
\end{Remark}

\section[The cohomology $H^1\big(\m_2^\lambda(p)\big)$ and $H^1_*\big(\m_2^\lambda(p)\big)$]{The cohomology $\boldsymbol{H^1\big(\m_2^\lambda(p)\big)}$ and $\boldsymbol{H^1_*\big(\m_2^\lambda(p)\big)}$}\label{section3}

In this short section we compute, for $p\ge 5$, both the ordinary and restricted 1-cohomology spaces $H^1(\m_2^\lambda(p))$ and $H^1_*(\m_2^\lambda(p))$, and in particular we show that these spaces are equal.

\begin{Theorem}\label{1-coho}
If $p\ge 5$ and ${\bf \lambda} \in \F^p$, then \[H^1\big(\m_2^\lambda(p)\big)=H^1_*\big(\m_2^\lambda(p)\big)\] and the classes of $\big\{e^1,e^2\big\}$ form a basis.
\end{Theorem}

\begin{proof} Easily, the differential (\ref{1-coboundary}) has a kernel spanned by $\big\{e^1,e^2\big\} $, and $d^0=0$, so that
 \[H^1\big(\m_2^\lambda(p)\big)\cong\F e^1\oplus \F e^2.\]

As for any restricted Lie algebra, $H^1_*\big(\m_2 ^\lambda (p)\big)$ consists of those ordinary cohomology classes $[\psi]\in H^1 \big(\m_2^\lambda (p)\big)$ for which $\ind^1(\psi)=0$ (see \cite[Theorem~2.1]{Hochschild1954} or \cite[Theorem~2]{EvansFuchs2008}). If $\psi =\sum_{k=1}^p \mu_ke^k$ is any ordinary cocycle, then $\mu_p=0$ ($p\ge 5$) so that for any
 $g\in \m_2^\lambda (p)$, we have
 \[\ind^1(\psi)(g)=\psi\big(g^{[p]}\big)=\mu_p\left (\sum_{k=1}^p \alpha_k^p\lambda_k\right )=0,\] and hence $H^1_*\big(\m_2^\lambda(p)\big)=H^1\big(\m_2^\lambda(p)\big)$.
\end{proof}

\begin{Remark} An alternate proof is the following
 \[
 H^1\big(\m_2^{\lambda}(p)\big)\cong \big(\m_2^{\lambda}(p)/\big[\m_2^{\lambda}(p), \m_2^{\lambda}(p)\big]\big)',
 \]
 where $\mathfrak g'=\Hom_\F(\mathfrak g,\F)$ denotes the vector space dual of $\mathfrak g$ and
 \[
 H^1_*\big(\m_2^{\lambda}(p)\big) \cong \big(\m_2^{\lambda}(p)/\big(\big[\m_2^{\lambda}(p),\m_2^{\lambda}(p)\big]+
 \spn\big(\m_2^{\lambda}(p)^{[p]}\big)\big)\big)'
 \]
 (see \cite[Proposition~2.7]{Feldvoss1991}).

In particular, since $\spn\big(\m_2^{\lambda}(p)^{[p]}\big)\subseteq \F e_p \subseteq\big[\m_2^{\lambda}(p),\m_2^{\lambda}(p)\big]$, the ordinary and the restricted 1-cohomology spaces coincide.
\end{Remark}

\begin{Remark}\label{1-cobndbasis} For $p\ge 5$, formula (\ref{1-coboundary}) shows that for $1\le k\le p$, ${\rm d}^1(e_k)=e^{1,k-1}+e^{2,k-2}$. For $k\ge 3$, let
 \begin{gather*}
 \phi_k={\rm d}^1\big(e^k\big)=e^{1,k-1}+e^{2,k-2}.
 \end{gather*}
 The set $\{\phi_3,\phi_4,\dots, \phi_p\}$ is a basis for the image ${\rm d}^1\big(C^1\big(\m_2^\lambda(p)\big)\big)$.
\end{Remark}

\section[The cohomology $H^2\big(\m_2^\lambda(p)\big)$ and $H^2_*\big(\m_2^\lambda(p)\big)$]{The cohomology $\boldsymbol{H^2\big(\m_2^\lambda(p)\big)}$ and $\boldsymbol{H^2_*\big(\m_2^\lambda(p)\big)}$}\label{section4}

\subsection{Ordinary cohomology}\label{section4.1}

\begin{Lemma}\label{ordinary2kernel} Let $p>5$. For $3\le k\le 2p-1$, let ${\rm d}^2_k\colon C^2_k\big(\m_2^\lambda(p)\big)\to C^3_k\big(\m_2^\lambda(p)\big)$ denote the
restriction of the differential ${\rm d}^2\colon C^2\big(\m_2^\lambda(p)\big)\to C^3\big(\m_2^\lambda(p)\big)$ to the $k$th graded component $C^2_k\big(\m_2^\lambda(p)\big)$ of $C^2\big(\m_2^\lambda(p)\big)$. Then $\ker \big({\rm d}^2_k\big)=0$ for $k\ge p+2$, and for $3\le k\le p+1$, basis elements of $\ker \big({\rm d}^2_k\big)$ are listed in
 the following table:
\begin{table}[h!]\centering
 \begin{tabular}{|l|l|} \hline
 $k=3$ & $\phi_3=e^{1,2}$ \tsep{1pt}\\
 \hline
 $k=4$ & $\phi_4=e^{1,3}$\tsep{1pt}\\
 \hline
 $k=5$ & $e^{1,4}$, $ e^{2,3}$\tsep{1pt}\\
 \hline
 $k=6$ & $\phi_6=e^{1,5}+e^{2,4}$\tsep{1pt}\\
 \hline
 $k=7$& $\phi_7=e^{1,6}+e^{2,5}$, $e^{1,6}+e^{3,4}$\tsep{1pt}\\
 \hline
$ 8\le k\le p+1$ & $\phi_k=e^{1,k-1}+e^{2,k-2}$\tsep{1pt}\\
 \hline
 \end{tabular}
\end{table}
\end{Lemma}

\begin{proof} A direct calculation using (\ref{2-coboundary}) shows that $d^2_k=0$ for $k=3,4,5$, ${\rm d}^2_6\big(\sigma_{1,5}e^{1,5} +\sigma_{2,4}e^{2,4}\big)=(-\sigma_{1,5}+\sigma_{2,4})e^{1,2,3}$ and ${\rm d}^2_7\big(\sigma_{1,6}e^{1,6} +\sigma_{2,5}e^{2,5}+\sigma_{3,4}e^{3,4}\big) =(-\sigma_{1,6} +\sigma_{2,5} +\sigma_{3,4})e^{1,2,4}$.

If $k\ge p+2$ and $\phi=\sum \sigma_{i,j} e^{i,j}\in C^2_k\big(\m_2^\lambda(p)\big)$, then~(\ref{2-coboundary}) implies that ${\rm d}^2_k(\phi)$ will contain the
 terms
 \[\sum_{\substack{k-p\le i< s(k)\\ i+j=k}} (\sigma_{i,j}+\sigma_{i+1,j-i})e^{1,i,j-1},\] where $s(k)=k/2-1$
 if $k$ is even and $s(k)=(k-1)/2$ if $k$ is odd. If, in addition, $\phi$ is a cocycle, then we have the system of equations
 \begin{gather} \label{system}
 \sigma_{i,j}+\sigma_{i+1,j-i} = 0,
 \end{gather}
where $k-p\le i< s(k)$ and $i+j=k$. Note that every coefficient $\sigma_{i,j}$ of $\phi$ occurs in the system~(\ref{system}). Now, if $k<2p-2$ is even, then ${\rm d}^2_k(\phi)$ contains exactly one term with $e^{2,k/2-2,k/2}$, and hence $\sigma_{k/2-2,k/2+2}=0$. The system~(\ref{system}) then implies that $\sigma_{i,j}=0$ for all $i$, $j$ and hence $\phi=0$. Likewise, if $k<2p-1$ is odd, then $e^{2,(k-1)/2-1,(k-1)/2}$ occurs once in ${\rm d}^2_k(\phi)$ forcing $\sigma_{(k-1)/2-1,(k-1)/2+2}=0$, and hence $\phi=0$. We can use (\ref{2-coboundary}) to directly check that $\ker \big({\rm d}^2_{2p-2}\big)=\ker \big({\rm d}^2_{2p-1}\big)=0$ so that $\ker \big({\rm d}^2_k\big)=0$ for
 $k\ge p+2$.

Finally, if $8\le k\le p+1$, then ${\rm d}^2_k(\phi)$ will also contain the terms
 \[\sum_{\substack{2\le i< s(k)\\ i+j=k}}
 (\sigma_{i,j}+\sigma_{i+1,j-i})e^{1,i,j-1}\] in addition to the term $-\sigma_{1,k-1}e^{1,2,k-3}$. If $\phi$ is a cocycle, then we
 have the system of equations
 \begin{gather}
 -\sigma_{1,k-1}+\sigma_{2,k-2}+\sigma_{3,k-3} =0,\nonumber\\
 \sigma_{i,j}+\sigma_{i+1,j-i} = 0, \qquad 3\le i < s(k).\label{system2}
 \end{gather}
 The same argument used for $k\ge p+2$ above shows that the system~(\ref{system2}) implies $\sigma_{i,j}=0$ for $i\ge 3$. Therefore $\sigma_{1,k-1}=\sigma_{2,k-2}$ and $\phi_k=e^{1,k-1}+e^{2,k-2}$ spans $\ker \big({\rm d}^2_k\big)$.
\end{proof}

\begin{Theorem}\label{ordinary2}
 If $p=5$, then
 \[\dim\big(H^2\big(\m_2^\lambda(5)\big)\big)=3\] and the
 cohomology classes of the cocycles
 $\big\{e^{1,4}, e^{1,5}+e^{2,4}, e^{2,5}-e^{3,4}\big\}$ form a basis.

 If $p>5$, then
 \[\dim\big(H^2\big(\m_2^\lambda(p)\big)\big)=3\] and the
 cohomology classes of the cocycles $\big\{e^{1,4}, e^{1,6}+e^{3,4}, e^{1,p}+e^{2,p-1}\big\}$ form a basis.
\end{Theorem}

\begin{proof} If $p=5$, then the results of Lemma~\ref{ordinary2kernel} still hold except for $k=7,8$ where we
 have
 ${\rm d}^2_7\big(\sigma_{2,5}e^{2,5}+\sigma_{3,4}e^{3,4}\big)=\big(\sigma_{2,5}+\sigma_{3,4}\big)e^{1,2,4}$
 and $\ker\big({\rm d}^2_8\big)=0$. It follows that
 \[\big\{e^{1,2}, e^{1,3}, e^{1,4}, e^{2,3},
 e^{1,5}+e^{2,4},e^{2,5}-e^{3,4}\big\}\] is a basis for $\ker
 \big({\rm d}^2\big)$. We can replace $e^{2,3}$ with $\phi_5=e^{1,4}+e^{2,3}$ in
 this basis so that, by Remark~\ref{1-cobndbasis}, the classes of
 $\big\{e^{1,4}, e^{1,5}+e^{2,4}, e^{2,5}-e^{3,4}\big\}$ form a basis for
 $H^2\big(\m_2^\lambda(5)\big)$.

 If $p > 5$, then Lemma~\ref{ordinary2kernel} gives a basis for $\ker \big({\rm d}^2\big)$. We can again replace $e^{2,3}$ with $\phi_5=e^{1,4}+e^{2,3}$ in this basis so that, by Remark~\ref{1-cobndbasis}, the classes of $\big\{e^{1,4}, e^{1,6}+e^{3,4}, e^{1,p}+e^{2,p-1}\big\}$ form a basis for
 $H^2\big(\m_2^\lambda(p)\big)$.
\end{proof}

\subsection[Restricted cohomology for $\lambda=0$]{Restricted cohomology for $\boldsymbol{\lambda=0}$}\label{section4.2}
If $\lambda=0$, then (\ref{induced2}) shows that $\ind^2=0$ so that
every ordinary 2-cocycle $\phi\in C^2\big(\m_2^0(p)\big)$ gives rise to a~restricted 2-cocycle $(\phi,\tilde \phi)\in
C^2_*\big(\m_2^0(p)\big)$. Moreover, in the case that $\phi={\rm d}^1(\psi)$ is a~1-coboundary, we can replace $\tilde\phi$ with $\ind^1(\psi)$ and
$\big(\phi,\ind^1(\psi)\big)=\big({\rm d}^1(\psi),\ind^1(\psi)\big)={\rm d}^1_*(\psi)$ is a~restricted 1-coboundary as well, and
${\rm d}^2_*(\phi,\tilde\phi)={\rm d}^2_*\big(\phi,\ind^1(\psi)\big)$ by Remark~\ref{changecomp2}. Finally, ${\rm d}^2_*\big(0,\overline e^k\big)=(0,0)$ for
all $1\le k\le p$, and the $\big(0,\overline e^k\big)$ are clearly linearly independent. Together these remarks prove the following

\begin{Theorem}\label{zerolambda} Let $\lambda=0$.
 If $p=5$, then
 \[\dim\big(H_*^2\big(\m_2^0(5)\big)\big)=8\] and the
 cohomology classes of the cocycles
 \[\big\{\big(0,\overline e^1\big),\big(0,\overline e^2\big),\big(0,\overline e^3\big),
 \big(0,\overline e^4\big), \big(0,\overline e^5\big), \big(e^{1,4},\tilde e^{1,4}\big),
 \big(\phi_6, \tilde \phi_6\big), \big(\xi, \tilde \xi\big)\big\}\] form a basis where
 $\xi=e^{2,5}-e^{3,4}$.

 If $p>5$, then
 \[\dim\big(H^2_*\big(\m_2^0(p)\big)\big)=p+3\] and the cohomology
 classes of
 \[\big\{\big(0,\overline e^1\big),\dots, \big(0,\overline e^p\big),\big( e^{1,4},\tilde
 e^{1,4}\big),\big(\eta,\tilde \eta\big),\big(\phi_{p+1},\tilde \phi_{p+1}\big)\big\}\]
 form a basis where $\eta=e^{1,6}+e^{3,4}$.
\end{Theorem}

\begin{Remark}\label{mostarezero} If $p\ge 5$, the maps $\tilde\phi_k$ are
 identically zero for $k<p+1$ because the $(p-1)$-fold bracket in~(\ref{starprop}) always gives a multiple of $e_p$ so that $\phi_k$
 vanishes on $e_p\wedge \m_2^\lambda(p)$ when $k<p+1$. This, in turn,
 implies that $\tilde\phi_k\in\Hom_{\rm Fr}\big(\m_2^\lambda(p),\F\big)$, and
 since $\tilde\phi_k(e_i)=0$ for all $i$, we have
 $\tilde\phi_k=0$. Likewise, $\tilde e^{1,4}=0$, $\tilde\eta=0$, and
 $\tilde\xi=0$ unless $p=5$. The restriction of $\phi_{p+1}$ to
 $e_p\wedge \m_2^\lambda(p)$ is equal to $e^{1,p}$ so that
 $\tilde\phi_{p+1}=\tilde e^{1,p}$. If $p=5$, then the restriction of
 $\xi$ to $e_5\wedge \m_2^\lambda(5)$ is equal to $e^{2,5}$ so
 $\tilde\xi=\tilde e^{2,5}$. We can then use~(\ref{starprop}) to give
 explicit descriptions for $\tilde\phi_{p+1}$ and $\tilde\xi$ (when $p=5$):
 \begin{gather*}
 \tilde\phi_{p+1}\left (\sum_{i=1}^p \alpha_ie_i\right )=\tilde
 e^{1,p}\left (\sum_{i=1}^p \alpha_ie_i\right
 ) =\alpha_1^{p-1}\alpha_2,\nonumber\\
 \tilde\xi\left (\sum_{i=1}^5 \alpha_ie_i\right )=\tilde
 e^{2,5}\left (\sum_{i=1}^5 \alpha_ie_i\right
 ) =\frac{1}{2}\alpha_1^{3}\alpha_2^2.
 \end{gather*}
\end{Remark}

\begin{Remark} The dimensions in Theorem~\ref{zerolambda} can also be deduced from Theorems~\ref{1-coho} and~\ref{ordinary2} and the six-term exact sequence in~\cite{Hochschild1954} precisely as in~\cite[Remark~4]{EvFi2019}.
\end{Remark}

\subsection[Restricted cohomology for $\lambda\ne 0$]{Restricted cohomology for $\boldsymbol{\lambda\ne 0}$}\label{section4.3}

If $\phi=\sum\sigma_{ij}e^{i,j}$ and $(\phi,\omega)\in C^2_*\big(\m_2^\lambda\big)$, then (\ref{induced2}) shows that
\[\ind^2(\phi,\omega)(e_j,e_i)=\lambda_i\sigma_{jp}.\]
Therefore, if $\lambda\ne 0$, then
${\rm d}^2_*(\phi,\omega)=\big({\rm d}^2\phi,\ind^2(\phi,\omega)\big)=(0,0)$ if and only
if ${\rm d}^2\phi=0$ and $\sigma_{1p}=\sigma_{2p}=\cdots = \sigma_{p-1 p}=0$.

\begin{Theorem}\label{nonzerolambda}
 If $p=5$, then
 \[\dim\big(H^2_*\big(\m_2^\lambda(5)\big)\big)=6\] and the cohomology
 classes of
 \[\big\{\big(0,\overline e^1\big), \big(0,\overline e^2\big),\big(0,\overline
 e^3\big),\big(0,\overline e^4\big), \big(0,\overline e^5\big), \big( e^{1,4},\tilde
 e^{1,4}\big)\big\}\] form a basis.

 If $p>5$, then
 \[\dim\big(H^2_*\big(\m_2^\lambda(p)\big)\big)=p+2\] and the cohomology
 classes of
 \[\big\{\big(0,\overline e^1\big),\dots, \big(0,\overline e^p\big),\big( e^{1,4},\tilde
 e^{1,4}\big),\big(\eta,\tilde \eta\big)\big\}\] form a basis where $\eta=e^{1,6}+e^{3,4} $.
\end{Theorem}

\section{One-dimensional central extensions}\label{section5}

One-dimensional central extensions $E=\g\oplus\F c$ of an ordinary Lie algebra $\g$ are parameterized by the cohomology group $H^2(\g)$ \cite[Chapter~1, Section~4.6]{FuchsBook}, and restricted one-dimensional central extensions of a restricted Lie algebra $\g$ with $c^{[p]}=0$ are parameterized by the restricted cohomology group $H^2_*(\g)$
\cite[Theorem~3.3]{Hochschild1954}. If $(\phi,\omega)\in C^2_*(\g)$ is a~restricted 2-cocycle, then the corresponding restricted one-dimensional central extension $E=\g\oplus\F c$ has Lie bracket and $[p]$-operation defined by
\begin{gather}
 [g,h] =[g,h]_{\g}+\phi(g\wedge h) c,\nonumber\\
 [g, c] = 0,\nonumber\\
 g^{[p]} = p^{[p]_{\g}}+\omega(g) c,\nonumber\\
 c^{[p]} = 0,\label{genonedimext}
\end{gather}
where $[\cdot,\cdot]_\g$ and $\cdot^{[p]_\g}$ denote the Lie bracket and $[p]$-operation in $\g$, respectively \cite[equations~(26) and~(27)]{EvansFuchs2008}. We can use (\ref{genonedimext}) to explicitly describe the restricted one-dimensional central extensions corresponding to the restricted cocycles in Theorems~\ref{zerolambda} and~\ref{nonzerolambda}. For the rest of this section, let $g=\sum \alpha_ie_i$ and $h=\sum\beta_ie_i$ denote two arbitrary elements of $\m_2^\lambda(p)$.

Let $E_k=\m_2^\lambda(p)\oplus \F c$ denote the one-dimensional restricted central extension of $\m_2^\lambda(p)$ determined by the cohomology class of the restricted cocycle $\big(0,\overline e^k\big)$. Then just as with~$\m_0^\lambda(p)$ and~$\V(p)$ (see \cite[Theorem~5.1]{EvFi2019} and \cite[Theorem~3.1]{EvFiPe2016}), the
$\big(0,\overline e^k\big)$ span a $p$-dimensional subspace of~$H^2_*$, and~(\ref{genonedimext}) gives the bracket and $[p]$-operation in~$E_k$:
\begin{gather*}
 [g,h] =[g,h]_{\m_2^\lambda(p)},\\
 [g, c] = 0,\\
 g^{[p]} = g^{[p]_{\m_2^\lambda(p)}}+ \alpha_k^p c,\\
 c^{[p]} = 0.
\end{gather*}

For restricted cocycles $(\phi,\tilde\phi)$ with $\phi\ne 0$, we summarize the corresponding restricted one-dimensional central extensions $E_{(\phi,\tilde\phi)}$ in the following tables. Everywhere in the tables, we omit the brackets $[g,c]=0$ and $[p]$-operation $c^{[p]}=0$ for brevity.

If $\lambda=0$, then there are three restricted cocycles
$(\phi,\tilde\phi)$ with $\phi\ne 0$ for a given prime
(Theorem~\ref{zerolambda}). We note that if $\lambda=0$, then
(\ref{p-op}) implies $g^{[p]_{\m_2^\lambda(p)}}=0$ for all
$g\in \m_2^\lambda(p)$.
\begin{table}[h!]\centering
\caption{Restricted one-dimensional central extensions with $\phi\ne 0$ and $\lambda=0$.}\vspace{1mm}
 \begin{tabular}{|l|l|}
 \hline
 \multicolumn{2}{|c|}{$p=5$}\\
 \hline
 $\big(e^{1,4},0\big)$ &
 $[g,h]=[g,h]_{\m_2^\lambda(p)}+(\alpha_1\beta_4-\alpha_4\beta_1) c$\tsep{2pt}\\
 & $g^{[p]} = 0$\\
 \hline
 $\big(\xi,\tilde\xi\big)$ &
 $[g,h]=[g,h]_{\m_2^\lambda(p)}+(\alpha_2\beta_5-\alpha_5\beta_2-\alpha_3\beta_4+\alpha_4\beta_3) c$\tsep{2pt}\\
 $\xi=e^{2,5}-e^{3,4}$ & $g^{[p]} = \frac{1}{2}\alpha_1^3\alpha_2^2c$\bsep{2pt}\\
 \hline
 $(\phi_6,\tilde\phi_6)$ &
 $[g,h]=[g,h]_{\m_2^\lambda(p)}+(\alpha_1\beta_5-\alpha_5\beta_1+\alpha_2\beta_4-\alpha_4\beta_2) c$\tsep{2pt}\\
 & $g^{[p]} = \frac{1}{2}\alpha_1^4\alpha_2c$\bsep{2pt}\\
 \hline
 \multicolumn{2}{|c|}{$p>5$}\\
 \hline
 $\big(e^{1,4},0\big)$ &
 $[g,h]=[g,h]_{\m_2^\lambda(p)}+(\alpha_1\beta_4-\alpha_4\beta_1) c$\tsep{2pt}\\
 & $g^{[p]} = 0$\\
 \hline
 $(\eta,0)$ &
 $[g,h]=[g,h]_{\m_2^\lambda(p)}+(\alpha_1\beta_6-\alpha_6\beta_1+\alpha_3\beta_4-\alpha_4\beta_3) c$\tsep{2pt}\\
 $\eta=e^{1,6}+e^{3,4}$ & $g^{[p]} = 0$\\
 \hline
 $(\phi_{p+1},\tilde\phi_{p+1})$ &
 $[g,h]=[g,h]_{\m_2^\lambda(p)}+(\alpha_1\beta_p-\alpha_p\beta_1+\alpha_2\beta_{p-1}-\alpha_{p-1}\beta_2) c$\tsep{2pt}\\
 & $g^{[p]} = \alpha_1^{p-1}\alpha_2c$\\
 \hline
 \end{tabular}
\end{table}

If $\lambda\ne 0$ and $p=5$, then the only restricted cocycle
$(\phi,\tilde\phi)$ with $\phi\ne 0$ is $(e^{1,4},0)$. If $p>5$, then
the restricted cocycles $(\phi,\tilde\phi)$ with $\phi\ne 0$ are
$(e^{1,4},0)$ and $(\eta,0)$ (Theorem~\ref{nonzerolambda}).
\begin{table}[h!]\centering
\caption{Restricted one-dimensional central extensions with $\phi\ne 0$ and $\lambda\ne 0$.}\vspace{1mm}
 \begin{tabular}{|l|l|}
 \hline
 \multicolumn{2}{|c|}{$p=5$}\\
 \hline
 $\big(e^{1,4},0\big)$ &
 $[g,h]=[g,h]_{\m_2^\lambda(p)}+(\alpha_1\beta_4-\alpha_4\beta_1) c$\tsep{2pt}\\
 & $g^{[p]} = g^{[p]_{\m_2^\lambda(p)}}$\\
 \hline
 \multicolumn{2}{|c|}{$p>5$}\\
 \hline
 $\big(e^{1,4},0\big)$ &
 $[g,h]=[g,h]_{\m_2^\lambda(p)}+(\alpha_1\beta_4-\alpha_4\beta_1) c$\tsep{2pt}\\
 & $g^{[p]} = g^{[p]_{\m_2^\lambda(p)}}$\\
 \hline
 $(\eta,0)$ & $[g,h]=[g,h]_{\m_2^\lambda(p)}+(\alpha_1\beta_6-\alpha_6\beta_1+\alpha_3\beta_4-\alpha_4\beta_3) c$\tsep{2pt}\\
 $\eta=e^{1,6}+e^{3,4}$ & $g^{[p]} = g^{[p]_{\m_2^\lambda(p)}}$\\
 \hline
 \end{tabular}
\end{table}

\subsection*{Acknowledgements}

The authors are grateful to Dmitry Fuchs for fruitful conversations, and the referees whose comments greatly improved the exposition of this paper.

\pdfbookmark[1]{References}{ref}
\LastPageEnding

\end{document}